\documentclass[11pt]{amsart}
\usepackage{amssymb, latexsym}
\theoremstyle{plain}
\newtheorem{theorem}{Theorem}
\newtheorem{corollary}{Corollary}
\newtheorem*{thm-cheb}{Theorem (Chebyshev)}
\newtheorem {lemma}{Lemma}
\newtheorem{proposition}{Proposition}

\newtheorem*{2'}{Theorem 2'}
\newtheorem*{3'}{Theorem 3'}

\theoremstyle{remark}

\newtheorem*{Remark 1}{Remark 1}
\newtheorem*{Remark 2}{Remark 2}
\newtheorem*{Remark 3}{Remark 3}
\newtheorem*{Remark 4}{Remark 4}

\numberwithin{equation}{section}

\begin{document}

\title[Infinite Limit of Random Permutations Avoiding  Three-Patterns  ]
{The Infinite  Limit of Random Permutations Avoiding  Patterns of Length Three }
\author{Ross G. Pinsky}

%\noindent  pinsky@math.technion.ac.il\ \ \ \ tel: 972-4-829-4083\ \ \  fax: 972-4-829-3388

\address{Department of Mathematics\\
Technion---Israel Institute of Technology\\
Haifa, 32000\\ Israel}
\email{ pinsky@math.technion.ac.il}

\urladdr{http://www.math.technion.ac.il/~pinsky/}

\subjclass[2000]{60C05, 60B10, 05A05} \keywords{pattern-avoiding permutation, random permutation, infinite permutation, pattern of length three}
\date{}

\begin{abstract}
For $\tau\in S_3$, let $\mu_n^{\tau}$ denote the uniformly random probability measure on the set
of $\tau$-avoiding permutations in $S_n$.
Let $\mathbb{N}^*=\mathbb{N}\cup\{\infty\}$ with an appropriate metric and denote by
 $S(\mathbb{N},\mathbb{N}^*)$ the compact metric space consisting  of functions
$\sigma=\{\sigma_i\}_{ i=1}^\infty$ from
$\mathbb{N}$ to $\mathbb{N}^*$  which
are injections when restricted to $\sigma^{-1}(\mathbb{N})$\rm; that is, if $\sigma_i=\sigma_j$, $i\neq j$, then
$\sigma_i=\infty$.
Extending permutations $\sigma\in S_n$
by defining $\sigma_j=j$, for $j>n$, we have $S_n\subset
S(\mathbb{N},\mathbb{N}^*)$.
For each $\tau\in S_3$, we study the limiting behavior of the measures $\{\mu_n^{\tau}\}_{n=1}^\infty$
on $S(\mathbb{N},\mathbb{N}^*)$.
We obtain partial results for the permutation $\tau=321$ and  complete results for the other five  permutations $\tau\in S_3$.

\end{abstract}
\maketitle

\section{Introduction and Statement of Results}
We recall the definition of pattern avoidance for permutations. Let $S_n$ denote the set of permutations of $[n]:=\{1,\cdots, n\}$.
If $\sigma=\sigma_1\sigma_2\cdots\sigma_n\in S_n$ and $\tau=\tau_1\cdots\tau_m\in S_m$, where $2\le m<n$,
then we say that $\sigma$ contains $\tau$ as a pattern if there exists a subsequence $1\le i_1<i_2<\cdots<i_m\le n$ such
that for all $1\le j,k\le m$, the inequality $\sigma_{i_j}<\sigma_{i_k}$ holds if and only if the inequality $\tau_j<\tau_k$ holds.
If $\sigma$ does not contain $\tau$, then we say that $\sigma$ \it avoids\rm\ $\tau$.
We consider here permutations on $S_n$ that avoid a pattern $\tau\in S_3$.
Denote by $S_n(\tau)$  the set of permutation in $S_n$ that avoid $\tau$. It is well-known  that
$|S_n(\tau)|=C_n$,  for all six permutations $\tau\in S_3$, where $C_n=\frac{\binom{2n}n}{n+1}$ is the $n$th
Catalan number \cite{B}.
Let $\mu_n^\tau$ denote the uniformly random probability measure on $S_n(\tau)$.
In this paper we investigate the limiting behavior of the probability measures $\mu_n^\tau$ as $n\to\infty$.
In the limit we will obtain a probability measure not on the set of permutations of $\mathbb{N}:=\{1,2,\cdots\}$,
but on a more general structure which we now describe.

Let $\mathbb{N}^*=\mathbb{N}\cup \{\infty\}$ with the metric $d_{N^*}(i,j)=\sum_{k=i}^{j-1}2^{-k}$, for $1\le i<j\le \infty$.
Denote by $S(\mathbb{N},\mathbb{N}^*)$ the set of functions
$\sigma=\{\sigma_i\}_{ i=1}^\infty$ from
$\mathbb{N}$ to $\mathbb{N}^*$ \it which
are injections when restricted to $\sigma^{-1}(\mathbb{N})$\rm; that is, if $\sigma_i=\sigma_j$, $i\neq j$, then
$\sigma_i=\infty$.
Let $S(\mathbb{N},\mathbb{N})\subset S(\mathbb{N},\mathbb{N}^*)$ denote the subset of injections
from $\mathbb{N}$ to $\mathbb{N}$,   let
 $S_{\text{sur}}(\mathbb{N},\mathbb{N}^*)\subset S(\mathbb{N},\mathbb{N}^*)$
 denote the subset of surjections from
$\mathbb{N}$ to $\mathbb{N}^*$, and let
$S_\infty\subset S(\mathbb{N},\mathbb{N})$
denote the set of bijections from $\mathbb{N}$ to $\mathbb{N}$,
  that is, the set of permutations of $\mathbb{N}$.

The space $S(\mathbb{N},\mathbb{N}^*)$ can be identified with
the countably infinite product $\mathbb{N}^*\times\mathbb{N}^*\cdots$.
Since $\mathbb{N}^*$ is a compact metric space, it follows that $S(\mathbb{N},\mathbb{N}^*)$
is also a compact metric space with the metric $D(\sigma,\tau):=\sum_{i=1}^\infty \frac{d_{\mathbb{N}^*}(\sigma_i,\tau_i)}{2^{i}}$.
For any $n\in\mathbb{N}$, we  identify the set $S_n$ of permutations of $[n]$ with the subset
$\{\sigma\in S_\infty:\sigma_j=j,j>n\}$.
Consequently, if $\mu_n$ is a probability measure on $S_n$, for each $n\in\mathbb{N}$, then
$\{\mu_n\}_{n=1}^\infty$ may be considered as a  sequence of probability measures on
the compact metric space $S(\mathbb{N},\mathbb{N}^*)$.  Thus, any such sequence has a subsequence converging
weakly to a probability measure on $S(\mathbb{N},\mathbb{N}^*)$.

If one uses the above framework to study the limit of the uniform probability measure on $S_n$,
then it is easy to show that the sequence of  measures converges weakly to the degenerate distribution
$\delta_{\infty^{(\infty)}}$
on the point $\infty^{(\infty)}\in S(\mathbb{N},\mathbb{N}^*)$,
where $\infty^{(\infty)}$ denotes the function $\sigma\in S(\mathbb{N},\mathbb{N}^*)$ satisfying
$\sigma_n=\infty$, for all $n\in\mathbb{N}$. On the other hand,
consider
the Mallows distribution  on $S_n$ with parameter $q>0$. This is the probability measure that gives to any
permutation $\sigma\in S_n$ a probability
proportional to $q^{\text{inv}(\sigma)}$, where $\text{inv}(\sigma)$
denotes the number of inversions
in the permutation $\sigma$; that is, $\text{inv}(\sigma)=|\{(i,j):1\le i<j\le n \ \text{and}\ \sigma_i>\sigma_j\}|$.
When $q=1$, the Mallows measure is just the uniform measure.
When $q\in(0,1)$, the Mallows measure favors permutations with few inversions, and when $q>1$, it favors
permutations with many inversions.
When $q>1$, the sequence of Mallows distributions converges weakly to $\delta_{\infty^{(\infty)}}$, but when $q\in(0,1)$, these distributions converge weakly to a nontrivial distribution
on $S(\mathbb{N},\mathbb{N}^*)$ which is in fact supported on the set of permutations $S_\infty$.
The form of this limiting distribution is regenerative. See
\cite{GO10,GO12} for the limiting behavior of the Mallows distribution,
and see \cite{PT} and references therein for more on the general theory of regenerative infinite permutations.

Since the limit of the Mallows distribution with $q\in(0,1)$  is a distribution on $S_\infty$,
the more general framework of $S(\mathbb{N},\mathbb{N}^*)$ is not needed there. However,
this more general framework   is necessary for our
 study of the limiting behavior of the measures $\{\mu_n^\tau\}_{n=1}^\infty$, for $\tau\in S_3$.
It will turn out that the
limiting distribution is trivial in two out of the six cases, while in
  three out of the other four cases,  the limiting distribution
has  a regenerative structure.
In order to describe this regenerative structure,
we will need to consider permutations of subsets $I\subset\mathbb{N}$ not as functions with a domain, but rather just
as images. We will call such  an object a \it permutation image \rm\
of $I$.
%For $I\subset\mathbb{N}$, let $\Sigma_I$ denote
%the set of permutation images  of $I$.
Thus,
for example, if $I=\{3,4,9\}$, then there are six permutation images of $I$, which we denote by $(3\ 4 \ 9), (3\ 9\ 4),(4\ 3\ 9), (4\ 9\ 3), (9\ 3 \ 4),(9\ 4 \ 3)$.
We will denote a generic permutation image of $I$ by $\sigma^{\text{im}}_I$.
We also  define $\infty^{(j)}$ to be the $j$-fold image of $\infty$: $\infty^{(j)}=\underbrace{(\infty\infty\cdots\infty)}_{j\ \text{times}}$,  $j\in\mathbb{N}$.
We  will use these permutation images to build functions in $S(\mathbb{N},\mathbb{N}^*)$.
For example, if $I_1=\{3,4,9\}$ and $I_2=\{20,22,24,26,28,30\cdots\}$, and if the permutation images $\sigma^{\text{im}}_{I_i}$, $i=1,2$,
are given by
$\sigma^{\text{im}}_{I_1}=(9\ 3 \ 4)$ and  $\sigma^{\text{im}}_{I_2}=(22\ 20\ 26\ 24\ 30\ 28\ \cdots)$,
then $\sigma:=\sigma^{\text{im}}_{I_1}*\sigma^{\text{im}}_{I_2}$
denotes the function in  $S(\mathbb{N},\mathbb{N})$ given by
$\sigma_1=9,\sigma_2=3,\sigma_3=4,\sigma_4=22,\sigma_5=20,\sigma_6=26,\cdots$,
while $\sigma=\infty^{(2)}*\sigma^{\text{im}}_{I_1}*\infty^{(1)}*\sigma^{\text{im}}_{I_2}$ denotes the function in
$S(\mathbb{N},\mathbb{N}^*)$ given by
$\sigma_1=\infty,\sigma_2=\infty,\sigma_3=9,\sigma_4=3,\sigma_5=4,\sigma_6=\infty,\sigma_7=22,\sigma_8=20,\sigma_9=26,\cdots$.

The mathematical description of our results in the propositions and theorems that follow looks a bit complicated, so we deem it  worthwhile to begin
with  a
verbal synopsis of the results. In what follows, a permutation image of a \it block\rm\ means a permutation
image of a set of consecutive numbers.
\medskip

\noindent \bf 1.\rm\ $\tau=123$: Weak convergence to the trivial distribution $\delta_{\infty^{(\infty)}}$.

\noindent \bf 2.\rm\ $\tau=132$: Weak convergence to the trivial distribution $\delta_{\infty^{(\infty)}}$.

\noindent \bf 3.\rm\ $\tau=312$: Weak convergence to a limiting distribution
which is supported on $S(\mathbb{N},\mathbb{N})-S_\infty$, and  whose structure is a concatenation  that
alternates uniformly random $312$-avoiding permutations images of random finite blocks of infinite expected length
with permutation images of random singletons,
each random singleton being the largest value smaller than the values in the preceding finite block permutation image.
The random finite blocks are obtained in a regenerative
fashion.

\noindent \bf 4.\rm\ $\tau=231$: Weak convergence to a limiting distribution
which is supported on $S_{\text{sur}}(\mathbb{N},\mathbb{N}^*)$, and  whose structure is a concatenation  which
alternates uniformly random $231$-avoiding permutations images of random finite contiguous blocks
with permutation images of the singleton $\infty^{(1)}$. The lengths of the contiguous random finite blocks  are IID,
have infinite expectation and are
obtained
in a regenerative fashion.

\noindent \bf 5.\rm\ $\tau=213$: Weak   convergence to a limiting distribution
which is supported on $S(\mathbb{N},\mathbb{N}^*)-S(\mathbb{N},\mathbb{N})-S_{\text{sur}}(\mathbb{N},\mathbb{N}^*)$,
and  whose structure
is a concatenation  which
alternates permutation images of blocks of  $\infty$ of random finite length
with permutation images of
singletons whose values increase along the concatenation. The
 values of the singletons are obtained in a regenerative fashion, and the lengths of the
blocks of $\infty$ are IID, have infinite expectation  and are obtained in  a regenerative fashion.

\noindent \bf 6.\rm\ $\tau=321$: Here we only have partial results. The limit of any weakly convergent subsequence is a
concatenation  of a  Geom$(\frac12)$ number of uniformly
random block irreducible (for the definition, see the paragraph preceding Lemma \ref{irr321}) 321-avoiding
permutations of
finite contiguous  blocks, the entire set of integers
 starting from 1 and ending at some random
$N$. The  blocks, whose lengths have infinite expectation, are obtained in a regenerative fashion.
  If in fact, the limit is in $S_\infty$, then the continuation $Z$ of the concatenation,
is supported on block irreducible 321-avoiding permutations of the infinite set  $\{N+1,\cdots\}$. Thus, the regenerative structure only
maintains itself for a finite length.
\medskip

\bf\noindent Remark.\rm\ Note that the supports of the limiting distributions in cases (3), (4) and (5) are all disjoint.

We now state our results in full.
\begin{proposition}\label{123and132}
\

\noindent i. Let $\tau=123$. Then
$\lim_{n\to\infty}\mu_n^\tau=\delta_{\infty^{(\infty)}}$.

\noindent ii. Let  $\tau=132$. Then
$\lim_{n\to\infty}\mu_n^\tau=\delta_{\infty^{(\infty)}}$.
\end{proposition}

To present the rest of the  results, we need to introduce some more definitions.
The distribution of the random variable $X$ defined below will play an important role in our results.
\begin{equation}\label{X}
P(X=n)=\frac{C_n}{2\cdot4^n},\ n=0,1,\cdots,
\end{equation}
where $C_n$ is the   $n$th Catalan number.

\noindent \bf Remark.\rm\ As is well-known \cite{P}, $\frac{1-\sqrt{1-4x}}{2x}=\sum_{n=0}^\infty C_nx^n$, for $|x|<\frac14$. Since $C_n\sim (\pi)^{-\frac12}4^n n^{-\frac32}$, if follows that the series converges for $x=\frac14$, and
$\sum_{n=0}^\infty C_n(\frac14)^n=2$. Thus, \eqref{X} does indeed define a distribution.
It also follows that $EX^p<\infty$ for $p\in(0,\frac12)$ but not for $p=\frac12$.

Let $Y$ denote a random variable with distribution Geom$(\frac12)$:
\begin{equation}\label{Y}
P(Y=n)=(\frac12)^n,\ n\in\mathbb{N}.
\end{equation}

Define
\begin{equation}\label{Tsums}
\begin{aligned}
&T_0^X=T_0^Y=0\\
&T_n^X=\sum_{j=1}^n X_j\ \ \ \            T_n^Y=\sum_{j=1}^n Y_j,\ n\in\mathbb{N},
\text{where}\ \{X_n\}_{n=1}^\infty\ \text{and}\ \{Y_n\}_{n=1}^\infty\ \text{are}\\
&\text{ mutually independent IID sequences with}\
X_1\ \text{distributed according to}\
 \eqref{X}\\
 & \text{and}\ Y_1\ \text{distributed according to}\ \eqref{Y}.
\end{aligned}
\end{equation}

We define pattern avoidance for permutation images in the obvious way; for example the permutation image $(5\ 3\ 9\ 1)$ is 123-avoiding, but is not 321-avoiding (because of the terms 5\ 3\ 1).
For fixed $\tau\in S_3$ and for all finite blocks $I\subset\mathbb{N}$,  define the random permutation images $\Pi^\tau_I$ of $I$ as follows:
\begin{equation}\label{Pi}
\begin{aligned}
&\Pi_I^\tau\ \text{is uniformly distributed over}\ \tau\text{-avoiding  permutation images of the}\\
&\text{ finite block}\
I\subset\mathbb{N}\
 \text{and}\
 \{P_I^\tau:I\subset \mathbb{N}, |I|<\infty\}\ \text{are independent}.
\end{aligned}
\end{equation}
\noindent \it Note on Notation:\rm\ In the sequel we will frequently use the following notation for
blocks: $[a,b]:=\{a,\cdots, b\}\subset\mathbb{N}$,
%and $[a,a-1]:=\emptyset$,
for $a,b\in\mathbb{N}$ with $a\le b$.

\begin{theorem}\label{312}
Let $\tau=312$.
Let $\{X_n\}_{n=1}^\infty,\{Y_n\}_{n=1}^\infty$ and $\{\Pi_I^{312}:I\subset\mathbb{N}\}$  be mutually independent random variables
with $\{\Pi_I^{312}:I\subset\mathbb{N}\}$ as in \eqref{Pi} and with
 $\{X_n\}_{n=1}^\infty$,  $\{Y_n\}_{n=1}^\infty$, $T_n^X, T_n^Y$ as in \eqref{Tsums}.
Then
$\lim_{n\to\infty}\mu_n^\tau$ is the distribution of the $S(\mathbb{N},\mathbb{N})-S_\infty$-valued random variable
$$
\begin{aligned}
&*_{n=1}^\infty\Pi^{312}_{[T_n^Y+T_{n-1}^X+1,T_n^Y+T_n^X]}*(T_n^Y+T_{n-1}^X):=\\
&\Pi^{312}_{[T^Y_1+1,T_1^Y+T_1^X]}*(T_1^Y)*\Pi^{312}_{[T_2^Y+T_1^X+1,T_2^Y+T_2^X]}*(T_2^Y+T_1^X)*\cdots.
\end{aligned}
$$
%where $\Pi^{312}_{[T_n^Y+T_{n-1}^{X-1}+1,T_n^Y+T_n^{X-1}]}$ is a uniformly distributed 312-avoiding permutation image of %the interval
%$\{T_n^Y+T_{n-1}^{X-1}+1,\cdots T_n^Y+T_n^{X-1}\}$.
\end{theorem}

\begin{theorem}\label{231}
Let $\tau=231$.
Let $\{X_n\}_{n=1}^\infty$ and $\{\Pi_I^{231}:I\subset\mathbb{N}\}$  be mutually independent random variables
with $\{\Pi_I^{231}:I\subset\mathbb{N}\}$ as in \eqref{Pi} and with
$\{X_n\}_{n=1}^\infty$ and  $T_n^X$ as in \eqref{Tsums}.
Then
$\lim_{n\to\infty}\mu_n^\tau$ is the distribution of the $S_{\text{sur}}(\mathbb{N},\mathbb{N}^*)$-valued random variable
\begin{equation*}
\begin{aligned}
*_{n=1}^\infty\Pi^{231}_{[T^X_{n-1}+1,T^X_n]}*\infty^{(1)}:=\Pi^{231}_{[1,T^X_1]}*\infty^{(1)}*\Pi^{231}_{[T_1^X+1,T_2^X]}*\infty^{(1)}*\cdots.
\end{aligned}
\end{equation*}

\end{theorem}

For the next result,
we will need some additional notation.
Define
\begin{equation}\label{T-hatsums}
\begin{aligned}
&T_0^{\hat X}=T_0^{Y^{(0)}}=T_0^{Y^{(1)}}=0\\
&T_n^{\hat X}=\sum_{j=1}^n \hat X_j\ \ \ \ T_n^{ Y^{(i)}}=\sum_{j=1}^n  Y^{(i)}_j,\ n\in\mathbb{N},\ i=0,1,\\
&\text{where}\ \{\hat X_n\}_{n=1}^\infty, \{ Y^{(0)}_n\}_{n=1}^\infty \ \text{and}\  \{ Y^{(1)}_n\}_{n=1}^\infty\
\text{are}\
\text{mutually independent IID}\\
&\text{sequences with}\  \hat X_1\stackrel{\text{dist}}{=}X+1,\ \text{where} \ X\ \text{is as in}\
 \eqref{X},\\
 & \text{and}\  Y_1^{(i)}\stackrel{\text{dist}}{=}Y, i=0,1,\ \text{where}\  Y\ \text{is as in}\ \eqref{Y}.
\end{aligned}
\end{equation}
Let
\begin{equation}\label{Ber12}
\chi_{0,1}\stackrel{\text{dist}}{=}\text{Ber}(\frac12):\ P(\chi_{0,1}=0)=P(\chi_{0,1}=1)=\frac12.
\end{equation}
For $J=\{J_n\}_{n=1}^\infty$, where $J_n\in \mathbb{N}$, and
$I=(i_1,i_2,\cdots )\subset\mathbb{N}$  an increasing sequence,
define
\begin{equation}\label{IJ*}
\infty^{(J)}*I:=*_{n=1}^\infty\infty^{(J_n)}*(i_n)=\infty^{(J_1)}*(i_1)*\infty^{(J_2)}*(i_2)*\cdots.
\end{equation}
\begin{theorem}\label{213}
Let $\tau=213$.
Let $\{X_n\}_{n=1}^\infty$,  $\{\hat X_n\}_{n=1}^\infty$, $\{ Y_n^{(i)}\}_{n=1}^\infty$, $i=0,1$,
and $\chi_{0,1}$ be mutually independent random variables with $\{X_n\}_{n=1}^\infty$ as in \eqref{Tsums},
with    $T_n^{\hat X}$,
$T_n^{ Y^{(i)}}, i=0,1$,
  as in \eqref{T-hatsums}, and with $\chi_{0,1}$ as in \eqref{Ber12}.
Then $\lim_{n\to\infty}\mu_n^\tau$ is the distribution  of the
$S(\mathbb{N},\mathbb{N}^*)-S(\mathbb{N},\mathbb{N})-S_{\text{sur}}(\mathbb{N},\mathbb{N}^*)$ -valued random variable
$$
\chi_{0,1}\cdot\big(\infty^{(J)}*I^{(1)}\big)+\big(1-\chi_{0,1})\cdot\big(\infty^{(J)}*I^{(0)}\big),
$$
where $\infty^{(J)}*I^{(i)}$ is as in \eqref{IJ*}, with
$$
J=  \{X_n\}_{n=1}^\infty
$$
and
$$
\begin{aligned}
&I^{(1)}=\cup_{n=0}^\infty[T_n^{ Y^{(1)}}+T^{\hat X}_{T_n^{ Y^{(2)}}}+1,T_{n+1}^{ Y^{(1)}}+T^{\hat X}_{T_n^{ Y^{(2)}}}],\\
&I^{(0)}=\cup_{n=0}^\infty[T_n^{ Y^{(1)}}+T^{\hat X}_{T_{n+1}^{ Y^{(2)}}}+1,T_{n+1}^{ Y^{(1)}}+T^{\hat X}_{T_{n+1}^{ Y^{(2)}}}].
\end{aligned}
$$
\end{theorem}
For the  final pattern, $\tau=321$, we need some more notation and another concept.
Let $I\subset\mathbb{N}$ be a (possibly infinite) block of  integers, and
let $S_{(I)}$ denote the set of permutations of the block $I$. (In this notation, $S_n=S_{([n])}$.)
Let $\sigma\in S_{(I)}$ and write $I$ generically as $I=\{j+i:0\le i<n^*\}$, where $n^*\in\mathbb{N}^*$.
If there does not exist a $k$ satisfying $0\le k<n^*$  and such that $\sigma$ maps $\{j,\cdots, j+k\}$ to itself, then we call
$\sigma$  a\it\ block irreducible\rm\ permutation in $S_{(I)}$.
Denote the set of 321-avoiding permutations in $S_{(I)}$ by $S_{(I)}(321)$,
%For a permutation $\sigma\in S_n$ and $I\subset[n]$, let $\sigma_I=\{\sigma_i: i\in I\}$.
%(\it Warning: don't confuse this with $\sigma^{\text{im}}_I$, which is a permutation image of $I$.)\rm\
%Define
%$$
%D^{321;n}_n=\{\sigma\in S_n(321): \sigma_{[k]}\neq[k],\ \text{for}\ 1\le k<n\}.
%$$
and denote by  $S^{\text{b-irr}}_{(I)}(321)$ the set of  block irreducible permutations in $S_{(I)}(321)$.
We will prove the following lemma.
\begin{lemma}\label{irr321}
Let $I=\{m+1,\cdots, m+j\}$, for some $m,j\in\mathbb{N}$. Then
\begin{equation}
|S^{\text{b-irr}}_{(I)}(321)|=C_{j-1},\ j\ge1.
\end{equation}
\end{lemma}
\bf\noindent Remark.\rm\ Of course, $|S_{(I)}(321)|=C_j$, for $I$ as in the lemma.
\medskip

Let $\mathcal{I}$ denote the class of all
 finite blocks $I\subset\mathbb{N}$.
Define the random permutations  $\{\Pi^{321;\text{b-irr}}_{(I)}\}_{I\in\mathcal{I}}$ as follows:
\begin{equation}\label{Piirr}
\begin{aligned}
&\Pi^{321;\text{b-irr}}_{(I)} \ \text{is uniformly distributed over the set}\  S^{\text{b-irr}}_{(I)}(321)\ \text{of 321-avoiding}\\
& \text{ block irreducible permutations  of}\ I\in\mathcal{I}\ \text{and}\
 \{\Pi_{(I)}^{321;\text{b-irr}}\}_{I\in\mathcal{I}}\ \text{are independent}.
\end{aligned}
\end{equation}

\begin{proposition}\label{321}
Let $\tau=321$.
Let $\{\hat X_n\}_{n=1}^\infty, Y$ and  $\{\Pi_{(I)}^{321;\text{b-irr}}:I\subset\mathcal{I}\}$  be mutually independent random variables
with $\{\Pi_{(I)}^{321;\text{b-irr}}:I\subset\mathcal{I}\}$ as in \eqref{Piirr},  with
 $\{\hat X_n\}_{n=1}^\infty$ and $T_n^{\hat X}$  as in \eqref{T-hatsums} and with $Y$ as in \eqref{Y}.
Then the distribution  of any   weakly converging subsequence of $\{\mu^{321}_n\}_{n=1}^\infty$
is the distribution of an $S(\mathbb{N},\mathbb{N}^*)$-valued random variable of the form
$$
\big(*_{n=0}^{Y-2}\Pi^{321;\text{b-irr}}_{([T^{\hat X}_n+1,T^{\hat X}_{n+1}])}\big)*Z,
$$
for some appropriate $Z$. If the limiting distribution is in fact
supported on $S_\infty$, then the random variable $Z$, conditioned on $Y=y$ and $T^{\hat X}_{y-1}=M$, is almost surely a 321-avoiding block irreducible permutation of
 the infinite set $\{M+1,M+2,\cdots\}$.
\end{proposition}

%\bf \noindent Conjecture.\rm\ $\{\mu_n^{321}\}_{n=1}^\infty$ converges to a limiting   distribution supported on $S_\infty$.
%\medskip

Note that in Theorems \ref{312}-\ref{213}, the length of  each segment in the regenerative structure
is distributed as $X+1$, and the length of the first $n$ segments is given by $T_n^X+n$. Thus, it is of interest
to determine the growth rate of $T_n^X$.

\begin{proposition}\label{stable}
\begin{equation}\label{stablelimit}
\lim_{n\to\infty}\frac{T_n^X}{n^2}\stackrel{\text{dist}}{=}Z,
\end{equation}
where $Z$ is the one-sided stable distribution with     stability parameter $\frac12$ and characteristic function
$$
\phi(t)=Ee^{-itZ}=\exp\big(-\frac{\sqrt2}2\thinspace| t|^\frac12(1+i\thinspace\text{\rm sgn}(t)\big).
$$
\end{proposition}

In section \ref{prelim} we will state and prove several preliminary facts that will be used in the proofs of the main results,
and we will prove Lemma \ref{irr321}.
The five sections that follow section \ref{prelim} give the proofs respectively of
 Proposition \ref{123and132},  Theorems \ref{312}-\ref{213} and Proposition \ref{321}.
In the final section we proof Proposition \ref{stable}.

\bf\noindent An important note regarding the proofs.\rm\ The same basic idea is used in the proofs of
Theorems \ref{312}-\ref{213} (via Lemma \ref{13middle} in section \ref{prelim}).
A variant of that idea is used for the proof of Proposition \ref{321} (via Lemma \ref{irr321}).    However, to write down a complete and entirely rigorous proof is extremely tedious and may well obscure the relative simplicity
of the ideas behind the proofs. Thus, for the proof of Theorem \ref{312}, we begin with a rather
 verbal explanation  of the proof, and
then   prove completely rigorously the first few steps of the proof. From this, it will be clear that
one can precede similarly to obtain the entire proof.
After that, for the proofs of Theorems \ref{231} and \ref{213} and Proposition \ref{321}, we will only give the rather verbal  explanation, the rigorous
proof following very similarly to that of Theorem \ref{312}. On the other hand, the proof of Proposition \ref{123and132} is
short and direct.

\section{Some Preliminary Results}\label{prelim}
We begin with the proof of  Lemma \ref{irr321}, which  appeared in the introductory section.

\noindent \it Proof of Lemma \ref{irr321}.\rm\
It suffices to prove the lemma for $S_j^{\text{b-irr}}(321)=S_{([j])}^{\text{b-irr}}(321)$.
For $1\le j\le n< \infty$, let $S_n^{\text{b-irr};j}(321)$ denote the set of permutations in $S_n(321)$ which map
$[j]$ to $[j]$ but do not map $[k]$ to $[k]$ for $1\le k<j$. (In this notation
$S_n^{\text{b-irr};n}(321)=S_n^{\text{b-irr}}(321)$.)
%For $\sigma\in S_n^{\text{b-irr};j}(321)$, we will say that $j$ is the length of the \it first irreducible block\rm\ in $\sigma$.
We have
\begin{equation}\label{breakdown}
|S_n(321)|=C_n=\sum_{j=1}^n|S_n^{\text{b-irr};j}(321)|,\ 1\le n<\infty.
\end{equation}
It is well known that a permutation in $S_n$ belongs to $S_n(321)$   if and only if it is composed of two increasing
subsequences \cite{B}.
Thus, $\sigma\in S_n^{\text{b-irr};j}(321)$ if and only if
$\sigma=\tau*\nu$, where $\tau\in S_j^{\text{b-irr}}(321)$ and
$\nu\in S_{([j+1,n])}(321)$, that is, $\nu$  is a 321-avoiding permutation of $[j+1,n]$.
Of course, the number of 321-avoiding permutations  of $[j+1,n]$ is
$C_{n-j}$. Thus, $|S_n^{\text{b-irr};j}(321)|=|S_j^{\text{b-irr}}(321)|C_{n-j}$.
Substituting this in \eqref{breakdown}
gives
\begin{equation}\label{breakdownagain}
C_n=\sum_{j=1}^n|S_j^{\text{b-irr}}(321)|C_{n-j},\ n\ge 1.
\end{equation}
On the other hand, the fundamental recurrence relation for Catalan numbers \cite{P} gives
\begin{equation}\label{Catrec}
C_n=\sum_{j=1}^nC_{j-1}C_{n-j},\ n\ge1.
\end{equation}
Equating \eqref{breakdownagain} and \eqref{Catrec} successively for $n=1,2,\cdots$ shows
that $|S_j^{\text{b-irr}}(321)|=C_{j-1}$, for all $j\ge1$.
\hfill $\square$

\noindent \bf Remark.\rm\ From the proof of the lemma, we obtain the following fact, which will be used later:
\begin{equation}\label{irrnj}
|S_n^{\text{b-irr};j}(321)|=C_{j-1}C_{n-j}.
\end{equation}

The following lemma states a well-known fact about permutations avoiding certain patterns of length three. For completeness, we provide the short proof.
\begin{lemma}\label{13middle}
For $1\le j\le n$,

\noindent i. $\mu_n^{312}(\sigma^{-1}_1=j)=\mu_n^{213}(\sigma^{-1}_1=j)=\frac{C_{j-1}C_{n-j}}{C_n}$;

\noindent ii. $\mu_n^{231}(\sigma^{-1}_n=j)=\mu_n^{132}(\sigma^{-1}_n=j)=\frac{C_{j-1}C_{n-j}}{C_n}$.

\end{lemma}
\begin{proof}
A 312-avoiding permutation $\sigma\in S_n$ has the property that all of the numbers in the positions to the left
of the position occupied by 1 are smaller than all of the numbers in the positions to the right
of the position occupied by 1. That is, if $\sigma^{-1}_1=j_1$, then $\{2,\cdots, j_1\}$ appear in the first
$j_1-1$ positions of $\sigma$ and $\{j_1+1,\cdots, n\}$ appear in the last $n-j_1$ positions of $\sigma$.
In fact then, it follows that a permutation $\sigma\in S_n$ satisfying $\sigma^{-1}_1=j$ will be 312-avoiding if and only if
$(\sigma_1,\cdots, \sigma_{j-1})$  is a 312-avoiding permutation image of $\{2,\cdots, j_1\}$ and
$(\sigma_{j+1},\cdots, \sigma_n)$ is a 312-avoiding permutation image of $\{j+1,\cdots, n\}$.
The proof of the lemma  for the case  $\mu_n^{312}$ now follows from the fact that there are $C_{j-1}$ 321-avoiding permutation images of
$\{2,\cdots, j\}$ and $C_{n-j}$  312-avoiding permutation images of
$\{j+1,\cdots, n\}$.

The proof for $\mu_n^{213}$ follows similarly, using the fact that a 213-avoiding permutation $\sigma\in S_n$ has the property that all of the numbers in the positions
to the left of the position occupied by 1 are larger than all of the numbers in the positions to the right of the position occupied by 1.
The proof for $\mu_n^{231}$ ($\mu_n^{132}$) follows similarly from the fact that a 231-avoiding (132-avoiding) permutation $\sigma\in S_n$ has the property that all of the numbers in the
positions to the left of the position occupied by $n$ are smaller (larger) than all of the numbers in the positions to the right of the position occupied by $n$.
\end{proof}
\begin{lemma}\label{Catalanlemma}
For $n\in\mathbb{N}$, let $\nu_n$ be the probability measure on $\mathbb{N}^*$ satisfying
$$
\begin{aligned}
&\nu_n(j)=\frac{C_{j-1}C_{n-j}}{C_n}, \ j\in[1,n];\\
&\nu_n(j)=0, \ j\in\mathbb{N}^*-[1,n].
\end{aligned}
$$
Define the probability measure $\nu_n^{\text{rev}}$ on $\mathbb{N}^*$ by
$$
\begin{aligned}
&\nu_n^{\text{rev}}(j)=\nu_n(n+1-j), \ j\in[n];\\
&\nu^{\text{rev}}_n(j)=0, \ j\in\mathbb{N}^*-[1,n].
\end{aligned}
$$
Then $\{\nu_n\}_{n=1}^\infty$ and $\{\nu_n^{\text{rev}}\}_{n=1}^\infty$ both converge weakly to the probability measure $\nu$ on $\mathbb{N}^*$ satisfying
$$
\begin{aligned}
&\nu(j)=C_{j-1}(\frac14)^j, \ j\in\mathbb{N};\\
&\nu(\infty)=\frac12.
\end{aligned}
$$
\end{lemma}
\bf \noindent Remark.\rm\ Note that $X+1$ has the distribution of $\nu(\cdot\thinspace|\mathbb{N})$, where $X$ is as in \eqref{X}.
\medskip

\begin{proof}
By symmetry, it is enough to prove the lemma for $\{\nu_n\}_{n=1}^\infty$. A direct calculation shows that for each fixed $j$, $\lim_{n\to\infty}\frac{C_{n-j}}{C_n}=(\frac14)^j$.
Thus, $\lim_{n\to\infty}\nu_n(j)=C_{j-1}(\frac14)^j$, for $j\ge1$.
As noted in the remark following \eqref{X},
$\sum_{n=0}^\infty C_n(\frac14)^n=2$. Thus, $\sum_{j=1}^\infty C_{j-1}(\frac14)^j=\frac12$.
This proves the lemma.
\end{proof}

\section{Proof of Proposition \ref{123and132}}\label{123and132pr}
\noindent \it Proof of  i.\rm\
For fixed $j,M\in\mathbb{N}$, we give an upper bound on $\mu_n^{123}(\sigma_j=M)$.
%=\frac{|S_n(123)|}{C_n}$.
To construct a permutation $\sigma\in S_n(123)$ satisfying $\sigma_j=M$, there are certainly no more
than $(n-1)\cdots (n-j+1)$ ways to choose the values of $\{\sigma_1,\cdots, \sigma_{j-1}\}$.
Having chosen $\{\sigma_1,\cdots, \sigma_{j-1}\}$,
there are at least $n-M-j+1$ values larger than $M$ among
the numbers $\{\sigma_{j+1},\cdots, \sigma_n\}$.
Since $\sigma_j=M$,
all the values larger than $M$  among $\{\sigma_{j+1},\cdots, \sigma_n\}$ must appear in decreasing order.
Thus, at least $n-M-j+1$ of the values among $\{\sigma_{j+1},\cdots, \sigma_n\}$ must appear in decreasing order.
So with regard to $n-M-j+1$ such values, the only choice we have is which $n-M-j+1$ spaces out of $n-j$ spaces
 to use for them.
Therefore, we conclude that
$$
\mu_n^{123}(\sigma_j=M)\le\frac1{C_n} (n-1)\cdots (n-j+1)\binom{n-j}{n-M-j+1}(M-1)!\le \frac{n^{j+M-2}}{C_n}.
$$
Thus, for any $j,L\in\mathbb{N}$,
$$
\lim_{n\to\infty}\mu_n^{123}(\sigma_j\le L)\le\lim_{n\to\infty} L\frac{n^{j+L-2}}{C_n}=0.
$$
From this it follows that the distribution of
any weak limit of $\{\mu_n^{123}\}_{n=1}^\infty$ must be supported on the singleton $\infty^{(\infty)}$.
\medskip

\noindent \it Proof of ii.\rm\
For fixed $j,M\in\mathbb{N}$, we give an upper bound on $\mu_n^{132}(\sigma_j=M)$.
To construct a permutation $\sigma\in S_n(132)$ satisfying $\sigma_j=M$, there are certainly no more
than $(n-1)\cdots (n-j+1)$ ways to choose the values of $\{\sigma_1,\cdots, \sigma_{j-1}\}$.
Having chosen $\{\sigma_1,\cdots, \sigma_{j-1}\}$,
there are at least $n-M-j+1$ values larger than $M$ among
the numbers $\{\sigma_{j+1},\cdots, \sigma_n\}$.
Since $\sigma_j=M$,
all the values larger than $M$  among $\{\sigma_{j+1},\cdots, \sigma_n\}$ must appear in increasing order.
Thus, at least $n-M-j+1$ of the values among $\{\sigma_{j+1},\cdots, \sigma_n\}$ must appear in increasing order.
So with regard to $n-M-j+1$ such values, the only choice we have is which $n-M-j+1$ spaces out of $n-j$ spaces
to use for them.
Therefore, we conclude that
$$
\mu_n^{132}(\sigma_j=M)\le\frac1{C_n} (n-1)\cdots (n-j+1)\binom{n-j}{n-M-j+1}(M-1)!\le \frac{n^{j+M-2}}{C_n}.
$$
The proof is now completed  as it was in part i.

\hfill $\square$
\section{Proof of Theorem \ref{312}}\label{312pr}
We will need the following additional notation. For a permutation image $\sigma^{\text{im}}_I=(i_1\ i_2 \ \cdots\ i_l)$
of a block $I=\{j+1,\cdots, j+l\}$, let $\sigma_I^{\text{im}}-j$ denote the permutation $\tau\in S_l$ given by
$\tau_k=i_k-j,\ k\in[l]$.
Also, for any $I\subset\mathbb{N}$,  let  $\Sigma^{\text{im}}_{I}$   denote the collection of all permutation images of $I$.

%The idea of the proof is not difficult, but the notation required to write down the proof in a precise matter may obscure that fact.
%Thus, we begin with a few rigorous lines of the proof, and then use them to give a heuristic sketch of the rest of the proof. After that, we return to rigor and write out the first
% steps of the proof, from which the regenerative continuation will be clear.
By Lemma \ref{13middle},
\begin{equation}\label{CC}
\mu_n^{312}(\sigma^{-1}_1=j_1)=\frac{C_{j_1-1}C_{n-j_1}}{C_n},\ j_1\in[1,n].
\end{equation}
From the proof of \eqref{CC} in Lemma \ref{13middle}, it   follows that
\begin{equation}\label{CCcond}
\mu_n^{312}(\sigma^{\text{im}}_{I_1}*(1)*\sigma^{\text{im}}_{I_2})|\sigma^{-1}_1=j_1)=\mu^{312}_{j_1-1}(\sigma^{\text{im}}_{I_1}-1)\mu^{312}_{n-j_1}(\sigma^{\text{im}}_{I_2}-j_1),
\end{equation}
where $1\le j_1\le n$, $\sigma^{\text{im}}_{I_1}$ is    a permutation image of  $I_1=[2,j_1]$ and
 $\sigma^{\text{im}}_{I_2}$ is a permutation  image of $I_2=[j_1+1, n]$.

As noted at the end of the first section, we first
give a rather verbal explanation  of the proof.
From \eqref{CC} and Lemma \ref{Catalanlemma} with  the remark following it, along with \eqref{X} and \eqref{Tsums},
it follows that as $n\to\infty$,  $\sigma^{-1}_1$ will be carried off to  $\infty$ with probability $\frac12$,  and  will converge
to the distribution $X_1+1$  with probability $\frac12$.  Consider the former case.
Let $\sigma^{-1}_1=j_1$ be very large. Akin to the proof of Lemma \ref{13middle}, since the first $j_1-1$ places constitute a 312-avoiding permutation
of $[2,j_1]$, it follows that from among these numbers, all the numbers in the positions to the left of $\sigma^{-1}_2$ are smaller than all the numbers in positions to the right of $\sigma^{-1}_2$.
Thus, the same reasoning as in \eqref{CC} gives
$\mu_n^{312}(\sigma^{-1}_2=j_2|\sigma^{-1}_1=j_1)=\frac{C_{j_2-1}C_{j_1-j_2-1}}{C_{j_1-1}},\ j_2\in[1, j_1-1]$. Thus, as $j_1\to\infty$, it follows that
$\sigma^{-1}_2$ will be carried off to $\infty$ with probability $\frac12$ and  will converge
to the distribution $X_1+1$ with probability $\frac12$. Continuing like this, eventually, we will arrive as some $m\in\mathbb{N}$
such that $\sigma^{-1}_1,\cdots, \sigma^{-1}_{m-1}$ were all carried off to $\infty$, but $\sigma^{-1}_m$ converges
to the distribution $X_1+1$. Note that the probability of this occurring at any specific $m$ is $(\frac12)^m$; that is, this occurrence time
has the distribution of $Y_1$, as in \eqref{Tsums}.
Thus, what we see so far is that the numbers $1,\cdots, Y_1-1$ have escaped to $\infty$, the number $Y_1$ is in position $X_1+1$, and by \eqref{CCcond},  the first $X_1$ positions
are occupied by a permutation image $\sigma^{\text{im}}_I$ of $I=[Y_1+1,Y_1+X_1]$ and this permutation image has the uniform distribution on 312-avoiding permutation images
of $[Y_1+1,Y_1+X_1]$. Stating this in  the notation of \eqref{Tsums} and \eqref{Pi}, we have that  the first $T^X_1+1$ positions
look like $\Pi^{312}_{[T^Y_1+1,T^Y_1+T^X_1]}*(T^Y_1)$. This is just as in the statement of the theorem. Now everything after position $T^X_1+1$ is iterated, with the smallest number
still available there being $T^Y_1+T^X_1+1$.
By the same reasoning, the first of these numbers that does not run off to $\infty$ will be $T^Y_2+T^X_1$, its position will be $T^X_2+2$ and in positions
$[T^X_1+2,T^X_2+1]$ will appear a uniformly random 312-avoiding permutation image of   $[T_2^Y+T_1^X+1,T_2^Y+T_2^X]$, that is, $\Pi^{312}_{[T_2^Y+T_1^X+1,T_2^Y+T_2^X]}$.
We now have the initial part of the limiting random variable being  $\Pi^{312}_{[T^Y_1+1,T^Y_1+T^X_1]}*(T^Y_1)*\Pi^{312}_{[T_2^Y+T_1^X+1,T_2^Y+T_2^X]}*(T^Y_2+T^X_1)$,
as in the theorem.

We now turn to the  rigorous proof.
Using
 Lemma \ref{Catalanlemma} and the remark following it, along with \eqref{CC} and \eqref{CCcond}, it follows that
\begin{equation}\label{end1}
\begin{aligned}
&\lim_{n\to\infty}\mu_n^{312}(\sigma^{\text{im}}_{I_1}*(1)*\Sigma^{\text{im}}_{[j_1+1,n]}))=\\
&\frac12P(X_1=j_1-1)\mu^{312}_{j_1-1}(\sigma^{\text{im}}_{I_1}-1)=
P( T^X_1=j_1-1, T^Y_1=1)\mu^{312}_{j_1-1}(\sigma^{\text{im}}_{I_1}-1),\\
& \text{for}\ j_1\in\mathbb{N} \ \text{and}\ \sigma^{\text{im}}_{I_1} \ \text{a permutation image of}\ I_1=[2, j_1],
\end{aligned}
\end{equation}
where $T^X_1$ and $T^Y_1$ are as in \eqref{Tsums}.

Repeating the  procedure  that yielded \eqref{CC} and \eqref{CCcond}, we have
\begin{equation}\label{CCagain}
\begin{aligned}
&\mu_n^{312}(\sigma^{-1}_{j_1+1}=j_2|\sigma^{-1}_1=j_1)=\frac{C_{j_2-j_1-1}C_{n-j_2}}{C_{n-j_1}},\ j_2\in[j_1+1, n];\\
&\mu_n^{312}(\sigma^{-1}_2=j_2|\sigma^{-1}_1=j_1)=\frac{C_{j_2-1}C_{j_1-j_2-1}}{C_{j_1-1}},\ j_2\in[1, j_1-1],
\end{aligned}
\end{equation}
and then we have
\begin{equation}\label{CCcondagain1}
\begin{aligned}
&\mu_n^{312}(\sigma^{\text{im}}_{I_1}*(1)*\sigma^{\text{im}}_{I_2}*(j_1+1)*\sigma^{\text{im}}_{I_3})|\sigma^{-1}_1=j_1, \sigma^{-1}_{j_1+1}=j_2)=\\
&\mu^{312}_{j_1-1}(\sigma^{\text{im}}_{I_1}-1)\mu^{312}_{j_2-j_1-1}(\sigma^{\text{im}}_{I_2}-j_1)
\mu^{312}_{n-j_2}(\sigma^{\text{im}}_{I_3}-j_2),\ 1\le j_1<j_2\le n,
\end{aligned}
\end{equation}
where   $\sigma^{\text{im}}_{I_1}$ is    a permutation image of $I_1=[2, j_1]$,
 $\sigma^{\text{im}}_{I_2}$ is a permutation  image of $I_2=[j_1+2, j_2]$
 and $\sigma^{\text{im}}_{I_3}$ is a permutation image of $I_3=[j_2+1, n]$, and we have
 \begin{equation}\label{CCcondagain2}
\begin{aligned}
&\mu_n^{312}(\sigma^{\text{im}}_{I_1}*(2)*\sigma^{\text{im}}_{I_2}*(1)*\sigma^{\text{im}}_{I_3})|\sigma^{-1}_1=j_1, \sigma^{-1}_{2}=j_2)=\\
&\mu^{312}_{j_2-1}(\sigma^{\text{im}}_{I_1}-2)\mu^{312}_{j_1-1-j_2}(\sigma^{\text{im}}_{I_2}-j_2-1)
\mu^{312}_{n-j_1}(\sigma^{\text{im}}_{I_3}-j_1),\ 1\le j_2<j_1\le n,
\end{aligned}
\end{equation}
where $\sigma^{\text{im}}_{I_1}$ is a permutation image of $I_1=[3,j_2+1]$, $\sigma^{\text{im}}_{I_2}$ is a permutation
image of $I_2=[j_2+2,j_1]$ and $\sigma^{\text{im}}_{I_3}$ is a permutation image
of $I_3=[j_1+1,n]$.

Using Lemma \ref{Catalanlemma} along with \eqref{CC},  \eqref{CCcondagain2} and the second equation in \eqref{CCagain}, we have
\begin{equation}\label{end2}
\begin{aligned}
&\lim_{M\to\infty}\lim_{n\to\infty}
\mu_n^{312}(\sigma^{\text{im}}_{I_1}*(2)*\Sigma^{\text{im}}_{[j_2+2,n]}, \ \sigma^{-1}_1\ge M)=\\
&\frac14P(X_1=j_2-1)\mu^{312}_{j_2-1}(\sigma^{\text{im}}_{I_1}-2)=\\
&P(T_1^X=j_2-1, T^Y_1=2)\mu^{312}_{j_2-1}(\sigma^{\text{im}}_{I_1}-2),\\
&\text{for}\ j_2\in\mathbb{N}\ \text{and}\ \sigma^{\text{im}}_{I_1}\ \text{a permutation image of}\ I_1=[3, j_2+1].
\end{aligned}
\end{equation}

Repeating the procedure  yet again, we have
\begin{equation}\label{CCyetagain}
\begin{aligned}
&\mu_n^{312}(\sigma^{-1}_{j_2+1}=j_3|\sigma^{-1}_1=j_1  ,\sigma^{-1}_{j_1+1}=j_2)=\frac{C_{j_3-j_2-1}C_{n-j_3}}{C_{n-j_3}},\
 j_3\in[j_2+1,\cdots, n],\ j_1<j_2;\\
%&\mu_n^{312}(\sigma^{-1}_2=j_2||\sigma^{-1}_1=j_1)=\frac{C_{j_2-1}C_{j_1-j_2-1}}{C_{j_1-1}},\ j_2\in[1,j_1-1],
\end{aligned}
\end{equation}
and then applying this to
 \eqref{CCcondagain2} we have
\begin{equation}\label{CCcondyetagain2}
\begin{aligned}
&\mu_n^{312}(\sigma^{\text{im}}_{I_1}*(3)*\sigma^{\text{im}}_{I_2}*(2)*\sigma^{\text{im}}_{I_3}*(1)*\sigma^{\text{im}}_{I_4})|\sigma^{-1}_1=j_1, \sigma^{-1}_2=j_2,
\sigma^{-1}_3=j_3)=\\
&\mu^{312}_{j_3-1}(\sigma^{\text{im}}_{I_1}-3)\mu^{312}_{j_2-1-j_3}(\sigma^{\text{im}}_{I_2}-j_3-2)
\mu_{j_1-1-j_2}^{312}(\sigma^{\text{im}}_{I_3}-j_2-1)
\mu^{312}_{n-j_1}(\sigma^{\text{im}}_{I_4}-j_1),\\
& 1\le j_3<j_2<j_1\le n,
\end{aligned}
\end{equation}
where $\sigma^{\text{im}}_{I_1}$ is a permutation image of $I_1=[4,j_3+2]$,
$\sigma^{\text{im}}_{I_2}$ is a permutation image of $I_2=[j_3+3,j_2+1]$,
$\sigma^{\text{im}}_{I_3}$ is a permutation image of $I_3=[j_2+2,j_1]$ and
$\sigma^{\text{im}}_{I_4}$ is a permutation  image of $I_4=[j_1+1,n]$.
Using Lemma \ref{Catalanlemma}  along with \eqref{CC}, the second equation in \eqref{CCagain}, \eqref{CCyetagain} and \eqref{CCcondyetagain2}, we have
\begin{equation}\label{end3}
\begin{aligned}
&\lim_{M\to\infty}\lim_{n\to\infty}\mu_n^{312}(\sigma^{\text{im}}_{I_1}*(3)*\Sigma^{\text{im}}_{[j_3+3,n]},\ \sigma^{-1}_1\ge M,\ \sigma^{-1}_2\ge M)=\\
&\frac18P(X_1=j_3-1)\mu^{312}_{j_3-1}(\sigma^{\text{im}}_{I_1}-3)=P(T_1^X=j_3-1,T_1^Y=3)\mu^{312}_{j_3-1}(\sigma^{\text{im}}_{I_1}-3),
\end{aligned}
\end{equation}
for $j_3\in\mathbb{N}$ and $\sigma^{\text{im}}_{I_1}$  a permutation image of $I_1=[4,j_3+2]$.

It is clear from \eqref{end1},\eqref{end2} and \eqref{end3} that if we continue in this vein we obtain
\begin{equation}\label{allends}
\begin{aligned}
&\lim_{M\to\infty}\lim_{n\to\infty}\mu_n^{312}(\sigma^{\text{im}}_{I_1}*(i)*\Sigma^{\text{im}}_{[j_i+i,n]},\ \sigma^{-1}_1\ge M,\cdots\ \sigma^{-1}_{i-1}\ge M)=\\
&(\frac12)^iP(X_1=j_i-1)\mu^{312}_{j_i-1}(\sigma^{\text{im}}_{I_1}-i)=P(T_1^X=j_i-1,T_1^Y=i)\mu^{312}_{j_i-1}(\sigma^{\text{im}}_{I_1}-i),
\end{aligned}
\end{equation}
for $j_i\in\mathbb{N}$ and $\sigma^{\text{im}}_{I_1}$  a permutation image of $I_1=[i+1,j_i+i-1]$.
This shows that a random variable whose distribution is that of  a weakly convergent subsequence of $\{\mu_n^{312}\}_{n=1}^\infty$ must
be of the form $\Pi^{312}_{[T^Y_1+1,T^Y_1+T^X_1]}*(T^Y_1)*Z$,
for some random $Z$ distributed on $S(\mathbb{N}-[1,T^X_1+1],\mathbb{N}^*)$.

We now need to continue and peel off the next component from $Z$. We just show the following step.
Using Lemma \ref{Catalanlemma} along with \eqref{CC},     \eqref{CCcondagain1} and the first equation in \eqref{CCagain}, we have
\begin{equation}
\begin{aligned}
&\lim_{n\to\infty}\mu_n^{312}(\sigma^{\text{im}}_{I_1}*(1)*\sigma^{\text{im}}_{I_2}*(1+j_1)*\Sigma_{[j_2+1,n]}))=\\
&\frac14P(X_1=j_1-1)P(X_2=j_2-j_1-1)\mu^{312}_{j_1-1}(\sigma^{\text{im}}_{I_1}-1)\mu^{312}_{j_2-j_1-1}(\sigma^{\text{im}}_{I_2}-j_1)=\\
&P(X_1=j_1-1,X_2=j_2-j_1-1,Y_1=1,Y_2=1)\mu^{312}_{j_1-1}(\sigma^{\text{im}}_{I_1}-1)\mu^{312}_{j_2-j_1-1}(\sigma^{\text{im}}_{I_2}-j_1),
\end{aligned}
\end{equation}
where   $\sigma^{\text{im}}_{I_1}$ is    a permutation image of $I_1=[2, j_1]$ and
 $\sigma^{\text{im}}_{I_2}$ is a permutation  image of $I_2=[j_1+2, j_2]$.
Continuing in this vein will give us
for all $(k_1,k_2)\in\mathbb{N}\times\mathbb{N}$,
\begin{equation}
\begin{aligned}
&\lim_{n\to\infty}\mu_n^{312}(\sigma^{\text{im}}_{I_1}*(k_1)*\sigma^{\text{im}}_{I_2}*(k_1+j_1-1+k_2)*\Sigma^{\text{im}}_{[k_1+k_2+j_2-1,n]}))=\\
&(\frac12)^{k_1+k_2}P(X_1=j_1-1)P(X_2=j_2-j_1-1)\mu^{312}_{j_1-1}(\sigma^{\text{im}}_{I_1}-1)\mu^{312}_{j_2-j_1-1}(\sigma^{\text{im}}_{I_2}-j_1)=\\
&P(X_1=j_1-1,X_2=j_2-j_1-1,Y_1=k_1,Y_2=k_2)\mu^{312}_{j_1-1}(\sigma^{\text{im}}_{I_1}-1)\mu^{312}_{j_2-j_1-1}(\sigma^{\text{im}}_{I_2}-j_1),
\end{aligned}
\end{equation}
where   $\sigma^{\text{im}}_{I_1}$ is    a permutation image of $I_1=[k_1+1,k_1+ j_1-1]$ and
 $\sigma_{I_2}$ is a permutation  image of $I_2=[k_1+j_1+k_2, k_1+k_2+j_2-2]$.\
This shows that a random variable whose distribution is that of a weakly convergent subsequence of $\{\mu_n^{312}\}_{n=1}^\infty$ must
be of the form $\Pi^{312}_{[T^Y_1+1,T^Y_1+T^X_1]}*(T^Y_1)*
\Pi^{312}_{[T^X_2+T^Y_1+1,T^X_2+T^Y_2]}*(T^Y_2+T^X_1)*Z$,
for some random $Z$ distributed on $S(\mathbb{N}-[1,T^X_2+2],\mathbb{N}^*)$.
The proof is completed by iterating  on this regenerative structure.
\hfill $\square$

\section{Proof of Theorem \ref{231}}\label{231pr}
 As noted at the end of the introductory section, we will give a rather verbal explanation of the  proof, the completely rigorous
proof following
via the same considerations and methods used in the proof of Theorem \ref{312}.
By Lemma \ref{13middle},
\begin{equation}\label{CC231}
\mu_n^{231}(\sigma^{-1}_n=j_1)=\frac{C_{j_1-1}C_{n-j_1}}{C_n},\ j_1\in[1,n].
\end{equation}
From the proof of \eqref{CC231} in Lemma \ref{13middle} it follows
that
\begin{equation}\label{CCcond231}
\mu_n^{231}(\sigma^{\text{im}}_{I_1}*(n)*\sigma^{\text{im}}_{I_2})|\sigma^{-1}_n=j_1)=\mu^{231}_{j_1-1}(\sigma^{\text{im}}_{I_1})\mu^{231}_{n-j_1}(\sigma^{\text{im}}_{I_2}-j_1+1),
\end{equation}
where $1\le j_1\le n$, $\sigma^{\text{im}}_{I_1}$ is    a permutation image of $I_1=[1,j_1-1]$ and
 $\sigma^{\text{im}}_{I_2}$ is a permutation  image of $I_2=[j_1, n-1]$.

From \eqref{CC231} and Lemma \ref{Catalanlemma} with  the remark following it, along with \eqref{X} and \eqref{Tsums},
it follows that as $n\to\infty$,  $\sigma^{-1}_n$ will be carried off to  $\infty$ with probability $\frac12$,  and  will converge
to the distribution $X_1+1$  with probability $\frac12$.  Consider the latter case. Then as $n\to\infty$,
the  position $\sigma^{-1}_n=j_1$ will converge in distribution to $X_1+1$, and by \eqref{CCcond231} the first $X_1$ positions will constitute
a uniformly random 231-avoiding  permutation of $[1, X_1]$. Thus, the initial segment of any weakly
convergent subsequence of $\{\mu_n^{231}\}_{n=1}^\infty$ looks like $\Pi^{231}_{[1,X_1]}*(\infty)=\Pi^{231}_{[1,T^X_1]}*(\infty)$.

Now consider the former case. Let $\sigma^{-1}_n=j_1$ be very large. The first $j_1-1$ positions
are occupied by a uniformly random 231-avoiding permutation of $[1,j_1-1]$.
Then in particular, the position of $j_1-1$ will satisfy
$\mu^{231}_n(\sigma^{-1}_{j_1-1}=j_2|\sigma^{-1}_n=j_1)=\frac{C_{j_2-1}C_{j_1-j_2-1}}{C_{j_1-1}}$,
for $j_2\in[1,j_1-1]$. Since $j_1$ is going to $\infty$ (as $n\to\infty$),
$\sigma^{-1}_{j_1-1}$ will be carried off to $\infty$ with probability $\frac12$ and will converge in distribution to
$X_1+1$
with probability $\frac12$. Consider the latter case. Then just as in the latter case in the previous paragraph, the
 initial segment of any weakly
convergent subsequence of $\{\mu_n^{231}\}_{n=1}^\infty$ will look like $\Pi^{231}_{[1,T^X_1]}*(\infty)$.
On the other hand, in the former case, we iterate the process we have just described. So far we have
assumed that the former case has prevailed twice.  Eventually, after say $i$ times in a row of the former case prevailing,
the latter case will finally prevail, and then as above it  will follow that the
initial segment of any weakly
convergent subsequence of $\{\mu_n^{231}\}_{n=1}^\infty$ looks like $\Pi^{231}_{[1,T^X_1]}*(\infty)$.
This process now regenerates on the rest of the domain, that is, on $[T^X_1+2,\infty)$, giving as the next piece, $\Pi^{231}_{[T^X_1+1,T^X_2]}*(\infty)$, as so on.
\hfill $\square$

\section{Proof of Theorem \ref{213}}\label{213pr}
As we noted at the end of the introductory section, we will give a rather verbal explanation of the  proof, the completely rigorous
proof following
via the same considerations and methods used in the proof of Theorem \ref{312}.
By Lemma \ref{13middle},
\begin{equation}\label{CC213}
\mu_n^{213}(\sigma^{-1}_1=j_1)=\frac{C_{j_1-1}C_{n-j_1}}{C_n},\ j_1\in[1,n].
\end{equation}
From the proof of \eqref{CC213} in Lemma \ref{13middle} it follows
that
\begin{equation}\label{CCcond213}
\mu_n^{213}(\sigma^{\text{im}}_{I_1}*(1)*\sigma^{\text{im}}_{I_2})|\sigma^{-1}_1=j_1)=\mu^{213}_{j_1-1}(\sigma^{\text{im}}_{I_1}-n+j_1-1)\mu^{213}_{n-j_1}(\sigma^{\text{im}}_{I_2}-1),
\end{equation}
where $1\le j_1\le n$, $\sigma^{\text{im}}_{I_1}$ is    a permutation image of $I_1=[n-j_1+2,n]$ and
 $\sigma^{\text{im}}_{I_2}$ is a permutation  image of $I_2=[2, n-j_1+1]$.
From \eqref{CC213} and Lemma \ref{Catalanlemma} with  the remark following it,
%along with \eqref{X} and \eqref{Tsums},
it follows that as $n\to\infty$,  with probability $\frac12$, $n-\sigma^{-1}_1$  will converge
in distribution to  $\hat X_1-1$, and   with probability $\frac12$, $\sigma^{-1}_1$ will converge in distribution to $X_1+1$.

Consider the latter case.
Then as $n\to\infty$,
the  position $\sigma^{-1}_1=j_1$ will converge in distribution to  $X_1+1$, and by \eqref{CCcond213}
the distribution of the  permutation image $\sigma^{\text{im}}_{I_1}$
of $I_1=[n-j_1+2,n]$ will converge to the degenerate distribution $\delta_{\infty^{(X_1)}}$.
Thus, in this case, the initial segment of any weakly convergent subsequence of $\{\mu_n^{213}\}_{n=1}^\infty$ looks like $\infty^{(X_1)}*(1)$.

Now consider the former case. By \eqref{CCcond213}, conditioned on
$\sigma_1^{-1}=j_1$, the final $n-j_1$ positions in the permutation are a random 213-avoiding permutation
image of $[2,n-j_1+1]$. Thus, since $n-\sigma^{-1}_1=n-j_1$ is converging in distribution to $\hat X_1-1$, and consequently
$\sigma^{-1}_1=j_1$ is converging in distribution to $\infty$,
it follows that the values $[1,\hat X_1]$ get swept away to $\infty$. Thus the support of any weakly convergent
subsequence of $\{\mu_n^{213}\}_{n=1}^\infty$ will be on functions in $S(\mathbb{N},\mathbb{N}^*-[1,\hat X_1])$.

Iterating the above scenarios, we see that with probability $\frac12$, the latter case
will prevail during the first $ Y_1^{(1)}$ iterations, then the former case will prevail
for the next $ Y^{(2)}_1$ iterations, then the latter case for the next $ Y_2^{(1)}$ iterations,
then the former for the next $ Y_2^{(2)}$ iterations, etc., while also with probability
$\frac12$,  the former case will prevail for the first $ Y^{(2)}_1$ iterations, then the latter
for the next $ Y^{(1)}_1$ iterations, etc.
These two possibilities, each with probability $\frac12$, are represented in the statement of the theorem
by the random variable $\chi_{0,1}$, with
$\chi_{0,1}=1$ if the first of these two  possibilities occurs.
Let's say that the first of these two possibilities occurs, the second possibility  being handled similarly.
Then the latter case prevails on the first $ Y_1^{(1)}$ iterations.
This results in the initial segment of any weakly convergent subsequence of $\{\mu^{213}\}_{n=1}^\infty$ looking
like $\infty^{(X_1)}*(1)*\infty^{(X_2)}*(2)*\cdots
\infty^{(X_{ Y^{(1)}_1})}*( Y^{(1)}_1)$.
After this, the former case prevails for $Y^{(2)}_1$ iterations. This causes the values
$[ Y^{(1)}_1+1, Y^{(1)}_1+T^{\hat X}_{Y^{(2)}_1}]$ to get swept out to $\infty$.
After this, the latter case prevails again for $Y^{(1)}_2$ iterations.
This results in the next segment of any weakly convergent subsequence looking like
$\infty^{(X_{ Y^{(1)}_1+1})}*( Y^{(1)}_1+T^{\hat X}_{Y^{(2)}_1}+1)*\cdots
\infty^{(X_{ Y^{(1)}_1+ Y^{(1)}_2})}*( Y^{(1)}_1+T^{\hat X}_{ Y^{(2)}_1}+ Y^{(1)}_2)$,
or equivalently, like
$\infty^{(X_{ Y^{(1)}_1+1})}*( Y^{(1)}_1+T^{\hat X}_{Y^{(2)}_1}+1)*\cdots
\infty^{(X_{T^{ Y^{(1)}}_2})}*(T^{ Y^{(1)}}_2+T^{\hat X}_{ Y^{(2)}_1})$.
In the notation of the theorem, we thus see in these two segments the beginning  of $\infty^{(J)}*I^{(1)}$,
revealed for $I^{(1)}$ up to
$\cup_{n=0}^1[T_n^{ Y^{(1)}}+T^{\hat X}_{T_n^{ Y^{(2)}}}+1,T_{n+1}^{ Y^{(1)}}+T^{\hat X}_{T_n^{ Y^{(2)}}}]=
[1, Y^{(1)}_1]\cup[ Y^{(1)}+T^{\hat X}_{ Y^{(2)}_1}+1,T^{ Y^{(1)}}_2+T^{\hat X}_{ Y^{(2)}_1}]$.
The above procedure now regenerates again and so on.
\hfill $\square$

\section{Proof of Proposition \ref{321}}\label{321pr}
As noted at the end of the introductory section, we will give a rather verbal explanation of the  proof, the completely rigorous
proof following
via the same considerations and methods used in the proof of Theorem \ref{312}.
Recall the definition of $S_n^{\text{b-irr};j}(321)$  from the proof of Lemma \ref{irr321}.
For $\sigma\in S_n(321)$, let $\mathcal{J}_n(\sigma)=\min\{j\ge1: \sigma\in S_n^{\text{b-irr};j}(321)\}$.
Then by \eqref{irrnj}, we have
\begin{equation}\label{key321a}
\mu^{321}_n(\sigma\in \mathcal{J}_n^{-1}(j))=\frac{C_{j-1}C_{n-j}}{C_n},\ 1\le j\le n, \ n\ge1.
\end{equation}
Also, by the considerations in the proof of Lemma \ref{irr321}, we have
\begin{equation}\label{key321b}
\begin{aligned}
&\mu^{321}_n(\tau*\nu^{\text{im}} |\sigma\in \mathcal{J}_n^{-1}(j))=\mu_j^{321;\text{b-irr}}(\tau)\mu_{n-j}^{321}(\nu^{\text{im}}-j),
\text{for}\ \tau\in S_j^{\text{b-irr}}(321)\\
&  \text{and}\ \nu^{\text{im}}\ \text{a 321-avoiding permutation image of}\ [j+1,n],
\end{aligned}
\end{equation}
where $\mu_j^{321;\text{b-irr}}$ denotes the uniformly probability measure on $S_j^{\text{b-irr}}(321)$.

From \eqref{key321a} and  Lemma \ref{Catalanlemma}, it follows
that
\begin{equation*}
\begin{aligned}
&\lim_{n\to\infty}\mu^{321}_n(\sigma\in \mathcal{J}_n^{-1}(j))=\frac12P(\hat X_1=j),\ j=1,2,\cdots;\\
&\lim_{M\to\infty}\lim_{n\to\infty}\mu^{321}_n(\sigma\in \mathcal{J}_n^{-1}\big([M,\infty)\big)=\frac12.
\end{aligned}
\end{equation*}
Using this with \eqref{key321b} shows that with probability $\frac12$, the distribution of any weakly convergent subsequence
of $\{\mu_n^{321}\}_{n=1}^\infty$ will begin with a segment whose distribution is that of
$\Pi^{321;\text{b-irr}}_{[1,\hat X]}$, and alternatively, with probability $\frac12$, if  a weakly convergent subsequence  converges to a limiting distribution on $S_\infty$, then that limiting distribution
is supported on permutations with no irreducible block. Using regeneration and iterating the above procedure proves the proposition.
\hfill $\square$

\section{Proof of Proposition \ref{stable}}\label{prstable}
We have $T_n^X=\sum_{j=1}^n X_j$, where  $\{X_n\}_{n=1}^\infty$ are IID with distribution given in \eqref{X}.
To prove the proposition, it suffices to show that $\lim_{n\to\infty}E\exp(-it\frac{T_n^X}{n^2})$ is equal to the characteristic function appearing in the statement of the proposition.
We have
\begin{equation}\label{charfunc}
E\exp(-it\frac{T_n^X}{n^2})=\big(E\exp(-i\frac t{n^2} X_1)\big)^n,
\end{equation}
and
\begin{equation}\label{withs}
E\exp(-is X_1)=\frac12\sum_{n=0}^\infty e^{-isn}\frac{C_n}{4^n}.
\end{equation}
By the remark after \eqref{X}, it follows that
$$
\frac{1-\sqrt{1-4z}}{2z}=\sum_{n=0}^\infty C_nz^n
$$
defines an analytic function for   $|z|<\frac14$, and that the equality continues to hold for $|z|=\frac14$,
where $\sqrt{w}=|w|^\frac12\exp(\frac12i\text{Arg}(w))$, for Re$(w)>0$ and Arg$(w)\in(-\frac\pi2,\frac\pi2)$.
Thus, from \eqref{withs} with $s=\frac t{n^2}$, we have
\begin{equation}\label{withtn2}
E\exp(-i\frac t{n^2} X_1)=\frac{1-\sqrt{1-\exp(-i\frac t{n^2}})}{\exp(-i\frac t{n^2})}.
\end{equation}
Writing  $1-\exp(-i\frac t{n^2})=1-\cos\frac t{n^2}+i\sin\frac t{n^2}$,     we see that
$$
1-\exp(-i\frac t{n^2})=\frac {t^2}{n^4}+i\frac t{n^2}+  O(\frac1{n^6}),\ \text{as}\ n\to\infty.
$$
Consequently,
\begin{equation}\label{complex}
\sqrt{1-\exp(-i\frac t{n^2})}=(1+o(1))\frac{|t|^\frac12}n\exp\big(i\thinspace\text{sgn}(t)(\frac\pi4+o(1)\big),\ \text{as}\ n\to\infty.
\end{equation}
From \eqref{charfunc}, \eqref{withtn2} and \eqref{complex}, we have
\begin{equation}
\begin{aligned}
&\lim_{n\to\infty}E\exp(-it\frac{T_n^X}{n^2})=\\
&\lim_{n\to\infty}\exp(i\frac tn)\Big(1-\frac1n(1+o(1))|t|^\frac12\exp\big(i\thinspace\text{sgn}(t)(\frac\pi4+o(1)\big)  \Big)^n=\\
&\exp\big(-|t|^\frac12\exp(i\thinspace\text{sgn}(t)\frac\pi4)\big)=\exp\big(-\frac{\sqrt2}2|t|^\frac12(1+i\thinspace\text{sgn}(t))\big).
\end{aligned}
\end{equation}

\end{document}